\documentclass[12pt]{amsart}
\usepackage{amsmath,amsfonts,amssymb,mathrsfs,amstext,amscd,latexsym,amsthm}
\usepackage[margin=1.2in]{geometry}
\usepackage{comment}
\usepackage{mathtools}
\usepackage{hyperref}
\usepackage{enumitem}
\usepackage{xcolor}

\def\labelitemi{--}
\def\ba #1\ea {\begin{align} #1\end{align}}
\def\bann #1\eann {\begin{align*} #1\end{align*}}
\def\ben #1\een {\begin{enumerate} #1\end{enumerate}}
\def\bi #1\ei {\begin{itemize}\renewcommand\labelitemi{--} #1\end{itemize}}
\theoremstyle{plain}
\newtheorem{thm}{Theorem}[section]
\newtheorem*{thm*}{Theorem}
\newtheorem{lem}[thm]{Lemma}

\theoremstyle{remark}
\newtheorem{rem}{Remark}
\newtheorem*{problem}{Problem}

\usepackage{cite}
\author[Y.~Zhao]{Yuhang Zhao}
\address{School of Mathematics, Nanjing University, Nanjing, 210093, P. R. of China}
\email{yuhangzhao@smail.nju.edu.cn}

\title{A gap theorem on closed self-shrinkers of the mean curvature flow}

\begin{document}
\begin{abstract}
		In this paper, we prove a pinching theorem for  $n-$dimensional closed  self-shrinkers of the mean curvature flow. If  the  squared norm of the second fundamental form of a closed self-shrinker of arbitrary codimension   satisfies: $
		| \vec{\uppercase\expandafter{\romannumeral2}} |^2 \le 1 +\frac{1}{10 \pi (n+2)}$, then it must be the standard sphere $S^{n}(\sqrt{n})$. This result may  provide some evidence for the open problem 13.76 in  \cite{andrews2022extrinsic}. 
	\end{abstract}
	
    \maketitle

	\section{Introduction}
	
	Let \( X : \Sigma^n \to \mathbb{R}^{n+p} \) be an \( n \)-dimensional submanifold in the \((n+p)\)-dimensional Euclidean space.
	Consider the position vector \( X(\cdot,t) \) evolving in the direction of the mean curvature \( \vec{H} \), then it gives rise
	to a solution to the mean curvature flow:
	\begin{align}
		X : \Sigma^n \times [0, T) \to \mathbb{R}^{n+p}, \  \frac{\partial X}{\partial t} = \vec{H}. 
	\end{align}
	
	It is interesting to investigate an important class of solutions to the mean curvature
	flow $(1.1)$, called self-shrinkers. An immersion  \( X : \Sigma^n \to \mathbb{R}^{n+p} \) is called a self-shrinker
	if it satisfies  the quasilinear elliptic system:
	\begin{align}
		\vec{H} = -X^\perp, \label{self-shrinker}
	\end{align}
	where \( \perp \) denotes the projection onto the normal bundle of $X( \Sigma^n )$.
	
	Self-shrinkers are shrinking homothetically under mean curvature flow (see, e.g. \cite{colding2012generic}), and they
	describe possible blow up behaviors  at a given singularity of the mean curvature flow. In the curve case, Abresch and Langer \cite{abresch1986normalized} gave a complete classification of all solutions
	to $\eqref{self-shrinker}$. These curves are called Abresch-Langer curves. In the hypersurface case, Huisken \cite{huisken1990asymptotic} showed that  closed self-shrinkers in $\mathbb{R}^{n+1}$ with nonnegative mean curvature are $S^{n}(\sqrt{n})$ . 
	Later, this was extended in \cite{huisken1993hypersurfaces} to the complete noncompact case with  nonnegative mean curvature, bounded second fundamental form and polynomial volume growth. For this result,  Colding and Minicozzi \cite{colding2012generic} showed that Huisken’s
	classification theorem still holds without the  boundedness assumption of the second fundamental form. 
    
    When the upper bound of the squared  norm  of the second fundamental form is a specific value, there is also a lot  of  progress. Le and Sesum \cite{le2011blow} proved that  complete self-shrinkers with polynomial volume growth in codimension $1$ and  $|\vec{\uppercase\expandafter{\romannumeral2}}|^2 < 1$ must be hyperplane. Later, this result was generalized by Cao and Li \cite{cao2013gap} in arbitrary codimension which stated that an $n-$dimensional complete self-shrinker with polynomial volume growth and $|\vec{\uppercase\expandafter{\romannumeral2}}|^2 \le 1$ must be a generalized cylinder. If the lower bound of the squared  norm  of the second fundamental form is  1, Ding and Xin \cite{ding2014rigidity} proved that an \( n \)-dimensional complete self-shrinker with polynomial volume growth in \( \mathbb{R}^{n+1} \)and
	$ 0 \leq |\uppercase\expandafter{\romannumeral2}|^2 - 1 \leq 0.022$  must be a round sphere or cylinder. Later, Xu and Xu \cite{lei2017chern} and  \cite{Hongwei18shrinker} imporved the constant on the right-hand side to $\frac{1}{21}$ and $\frac{1}{18}$, respectively. Under the assumption that \( |\uppercase\expandafter{\romannumeral2}|^2 \) is a constant, Cheng and Wei \cite{cheng2015gap} also proved that an \( n \)-dimensional complete self-shrinker with polynomial volume growth in \( \mathbb{R}^{n+1} \) and $ |\uppercase\expandafter{\romannumeral2}|^2  \le \frac{10}{7}$ must be one of the generalized cylinders.    
   
   Regarding other gap problems of the curvature, Cao, Xu and Zhao \cite{cao2014pinching} proved that   for complete self-shrinkers with  polynomial volume growth, if the squared norm of the tracefree 
	second fundamental form satisfies $|\vec{\uppercase\expandafter{\romannumeral2}} - \frac{\vec{H}}{n} \otimes g|^2 \le \frac{1}{2}$ and the  mean curvature is nowhere vanishing, then they must be  special generalized cylinders.

However, in the work of Ding and Xin \cite{ding2014rigidity}, Xu and Xu \cite{lei2017chern,Hongwei18shrinker},at first glance, the lower bound condition seems to be somewhat strong, as  there are many self-shrinkers  suggesting that the supremum of the squared norm of the second fundamental is strictly large than  $1$ and the infimum of the squared norm of the second fundamental is strictly smaller  than $1$, i.e. 
	\begin{align}
		\inf_{\Sigma} |\vec{\uppercase\expandafter{\romannumeral2}}|^2 < 1 < \sup_{\Sigma} |\vec{\uppercase\expandafter{\romannumeral2}}|^2. \label{eqlege}
	\end{align} 
For example, Drugan and Kleene \cite{drugan2017immersed} constructed infinitely many complete, immersed self-shrinkers with rotational symmetry	for each of the following topological types: the sphere $S^n$, the plane $\mathbb{R}^n$, the cylinder $\mathbb{R} \times S^{n-1}$ and the torus $S^1 \times S^{n-1}$. Also, many examples they constructed on topological $S^2$ satisfying $\eqref{eqlege}$ (see, e.g. the appendix of \cite{alencar2022hopf} by Silva and Zhou). In the noncompact case, $\Gamma \times \mathbb{R}^{n-1}$ also satisfies $\eqref{eqlege}$, where $\Gamma$ is an Abresch-Langer curve \cite{abresch1986normalized} that is not a circle.

	In fact, the condition `` $ 0 \leq |\uppercase\expandafter{\romannumeral2}|^2 - 1 $ '' on self -shrinkers is motivated  by $1^{\circ}$ the condition`` $0 \leq |\uppercase\expandafter{\romannumeral2}|^2 - n$ '' in Chern conjecture on  minimal hypersurfaces of $S^{n+1}$, $2^{\circ} $ the corresponding technique: calculating and integrating the laplacian of $|\nabla \uppercase\expandafter{\romannumeral2}|^2$, and $3^{\circ}$ the noncomact case: $\Gamma \times \mathbb{R}^{n-1}$ ( where $\Gamma$ is an Abresch-Langer curve \cite{abresch1986normalized} that is not a circle ). The corresponding work on Chern conjecture can be found in  \cite{peng1983minimal,yang1998chern,ding2011chern's,lei2017chern}.  Indeed, there are many compact minimal hypersurface in $S^{n+1}$ such that  $\max\limits_{\Sigma} |\uppercase\expandafter{\romannumeral2}|^2 $ is close to $n$ and $\min\limits_{\Sigma}|\uppercase\expandafter{\romannumeral2}|^2<n$, see Otsuki \cite{otsuki1970minimal}.  
	
	However, in the case of $2-$dimensional compact self-shrinkers in $\mathbb{R}^3$, there is an interesting 
	phenomenon that if its squared norm of the second fundamental form is strictly less than $\frac{3}{2}$, i.e. $ |\vec{\uppercase\expandafter{\romannumeral2}}|^2 < \frac{3}{2},$ then it must be the standard sphere $S^{2}(\sqrt{2})$ (see Lemma \ref{lemmotived}). Motivated by the result, we  naturally  consider  the following gap problem without the lower bound of the squared norm of the second fundamental form: 
    \begin{problem}
    \textit{Whether there exists a positive number $\eta $ (depending only on $n$)  such that  any $n-$dimensional compact self-shrinker in any codimension with \begin{align}
		|\vec{\uppercase\expandafter{\romannumeral2}}|^2 \le 1 +\eta 
	\end{align}
  is $S^{n}(\sqrt{n}) ?$}
  \end{problem}
  Here we point out that the compactness  assumption is necessary. In fact, by  the appendix of \cite {mantegazza2011lecture}, there  are many Abresch-Langer  curves  $\Gamma$ with curvature close to 1 so that  the squared norm of the second fundemental form of $\Gamma \times \mathbb{R}^{n-1}$ has the property that  $\max\limits_{\Sigma}|\vec{\uppercase\expandafter{\romannumeral2}}|^2 $ is close to 1.
 
 The main theorem of this paper gives  a  positive  answer to  this question :
	\begin{thm}\label{main}
		Let $X:\Sigma^{n} \to \mathbb{R}^{n+p} (n \ge 2)$ be an $n-$dimensional compact self-shrinker with the  squared norm of the second fundamental satisfying
		\begin{align}
			| \vec{\uppercase\expandafter{\romannumeral2}} |^2 \le 1 +\frac{1}{10\pi (n+2)}, \label{pinching}
		\end{align}
     then $X(\Sigma^{n})$ must be the standard sphere $S^{n}(\sqrt{n})$. 
 \end{thm}

    \begin{rem}
        The theorem implies that the assumption `` $ 0 \leq |\uppercase\expandafter{\romannumeral2}|^2 - 1 $ ''  can be removed in the compact case. This is a direct  improvement of \cite{cao2013gap,ding2014rigidity,lei2017chern,Hongwei18shrinker}. Moreover, our technique  is completely different from  that of \cite{ding2014rigidity,lei2017chern,Hongwei18shrinker}. In particular, their work does not include the case of higher codimension.   
        \end{rem}
   
    \begin{rem}
       If $X:\Sigma^{n} \to \mathbb{R}^{n+p} (n \ge 2)$ be an $n-$dimensional compact self-shrinker with $|X| = \text{constant}$ and   $| \vec{\uppercase\expandafter{\romannumeral2}} |^2 < \frac{5}{3}$, then $X(\Sigma^{n})$ must be $S^{n}(\sqrt{n})$. This  follows from Lemma \ref{meanmax},  Li and Li \cite{an1992intrinsic} and Simon \cite{simons1968minimal}. 
    \end{rem}

    \begin{rem}
       Colding, Ilmanen and Minicozzi \cite{colding2015rigidity} showed that
the generalized cylinders are rigid in the  sense that any self-shrinker which is
sufficiently close to one of the generalized cylinders on a large and compact set must
itself be a generalized cylinder. Using this, Guang \cite{guang2016self} proved  that any embedded self-shrinker with $p=1$ and $|\vec{\uppercase\expandafter{\romannumeral2}}|^2 \le 1 +\alpha $ must also be a generalized cylinder, where 
 $\alpha$ is an implicit positive number that depends on $n$ and the upper bound of its entropy. Lee \cite{lee2021compactness} also showed that the space of two-dimensional self-shrinkers of any codimension with the entropy strictly less than $2$ is compact, hence he got  the same gap theorem as Guang's under  this condition.
    \end{rem} 
\begin{rem}
  It is  interesting to know  here whether one can improve the gap to a constant independent of dimension  $n$ like the work of \cite{ding2014rigidity,lei2017chern,Hongwei18shrinker}. 
   \end{rem}
	To prove  Theorem \ref{main}, we first prove the following key gradient estimation on the squared norm of the  position vector. 
	
	\begin{thm}\label{maingradest}
		Let $X:\Sigma^{n} \to \mathbb{R}^{n+p}(n \ge 2)$ be an $n-$dimensional compact self-shrinker with the  squared norm of second fundamental form satisfying
		\begin{align}
			(n-1)^2|\vec{\uppercase\expandafter{\romannumeral2}}|_{\max} ^2 \leq (n+1)|X|^2_{\max} +2,\label{hypothesis}		
            \end{align}
		then we have 
		\begin{align}
			| \nabla^{\Sigma}  |X|^2 | \leq  2 B \sqrt{ |X|^2_{\max}-n} \sqrt{ |X|^2_{\max}- |X|^2},
            \end{align}
	where $$B=\sqrt{\frac{-n^2+n-2+(n+1)|X|^2_{\max} +(n-1)^2 |\vec{\uppercase\expandafter{\romannumeral2}}|_{\max}^2}{n(n-3)  +2|X|^2_{\max} }},$$
and  $| X |_{\max}=\max\limits_{\Sigma} |X|,|\vec{\uppercase\expandafter{\romannumeral2}}|_{\max}=\max\limits_{\Sigma}|\vec{\uppercase\expandafter{\romannumeral2}}|$.
\end{thm}

	The main idea of this theorem  is  to apply the Bochner formula to  $|\nabla^{\Sigma} |X|^2|^2$ and constructing the function $\eqref{g}$ to consider its maximum, which is inspired by P.Li's work  on the gradient estimate of eigenfunctions \cite{li2012geometric}. The form of the  estimate   we need is 
   \begin{align}
			\max_{\Sigma} g \leq C(\delta |X|^2_{\max} - n). \notag 
            \end{align} 
This is because if we have this estimate, then under the condition\eqref{pinching}, $g_{\max}$ is small enough to play an important role in the proof of the main theorem  when $\delta$ is close to 1. Of course, we also need to add the condition \eqref{hypothesis} to obtain this form. The details are provided in section \ref{sec:prelims}. 
  
   In  the  section \ref{mainproof}, we obtain a pointwise estimate of $|\vec{H}|^2$ and $|X|^2$ along  some certain geodesic by using \ref{maingradest}. Combining this with  the energy variation in the comparison  theorem, we can choose suitable upper bound on  $| \vec{\uppercase\expandafter{\romannumeral2}} |^2_{\max} -1$ to determine the upper bound of the distance function from some certain point ( where $|X|^2$  is maximal ). From this, we can obtain global estimates  of $|\vec{H}|^2$ and $|\vec{\uppercase\expandafter{\romannumeral2}} - \frac{\vec{H}}{n} \otimes g|^2 $ to prove  $X(\Sigma^{n}) = S^{n}(\sqrt{n})$  by the rigidity theorem of Huisken \cite{huisken1990asymptotic} and Theorem 1.6 of Cao, Xu, Zhao\cite{cao2014pinching}.   

{\bf Acknowledgement:} The author is grateful to professor Y.L.Shi for his help.
   \section{Preliminaries}
   Let \( X : \Sigma^n \to \mathbb{R}^{n+p} (n \ge 2) \) be an compact  connected \( n \)-dimensional submanifold of class $C^{\infty}$ in the \((n+p)\)-dimensional Euclidean space. Let $\langle, \rangle$ and $\cdot$ be the standard Euclidean metric and dot product, respectively; $\nabla^{\Sigma}$ be the covariant differentiation of the metric on $\Sigma^{n}$ induced by $X$;  $R_{ijkl}, R_{ij}$  be the components of the Riemannian curvature and the Ricci curvature in local orthonormal basic $\{e_{i}, 1\leq i\leq n\}$  respectively, and $R$ be the scalar curvature. Let $f$ be a smooth function on $\Sigma^n,$ $f_{i}, f_{ij}$ and $f_{ijk}$ be the components of the first, second and third order covariant derivatives of $f$ in local orthonormal basis $\{e_{i}, 1\leq i\leq n\},$ respectively;  $\Delta^{\Sigma} f,$ $\nabla^{\Sigma} f$ and $\nabla^{\Sigma} \nabla^{\Sigma} f$ be the laplacian, gradient  and hessian of $f$, respectively.  We use following notations: 
$$
|\nabla^{\Sigma} f|^{2}=\sum_{i}f_{i}^{2},\ \
|\nabla^{\Sigma} \nabla^{\Sigma} f|^{2}=\sum_{i,j}f_{ij}^{2}, \ \ \Delta^{\Sigma} f = \sum_{i} f_{ii}.
$$
In  a local normal basis, the Ricci identity  can be written as follows:
$$
f_{ijk}-f_{ikj}=f_{l}R_{lijk}.
$$
Then we have the following Bochner formula:
\begin{align}
			\frac{1}{2} \Delta^{\Sigma} |\nabla^{\Sigma} f |^2 = |\nabla^{\Sigma}  \nabla^{\Sigma}  f|^2 + \langle \nabla^{\Sigma} \Delta^{\Sigma} f, \nabla^{\Sigma} f\rangle + Ric(\nabla^{\Sigma} f, \nabla^{\Sigma}  f). \label{Lap2}
		\end{align}   
The Gauss equations are given by 
\begin{subequations}\label{eq:Gauss equation}
\ba
R_{i j k l}={}&\vec{\uppercase\expandafter{\romannumeral2}}_{i k} \cdot \vec{\uppercase\expandafter{\romannumeral2}}_{j l}-\vec{\uppercase\expandafter{\romannumeral2}}_{i l} \cdot \vec{\uppercase\expandafter{\romannumeral2}}_{j k} \label{eq:Riem} ,\\
R_{i k}={}&\Vec{H} \cdot \vec{\uppercase\expandafter{\romannumeral2}}_{i k} -\sum_{ j} \vec{\uppercase\expandafter{\romannumeral2}}_{i j} \cdot \vec{\uppercase\expandafter{\romannumeral2}}_{j k}\label{eq:Ric},\\
R={}&|\vec{H}|^2-|\vec{\uppercase\expandafter{\romannumeral2}}|^2. \notag
\ea
\end{subequations}
where $|\vec{\uppercase\expandafter{\romannumeral2}}|^2=\sum_{i, j} |\vec{\uppercase\expandafter{\romannumeral2}}_{i j}|^2 =\sum_{i, j} \vec{\uppercase\expandafter{\romannumeral2}}_{i j} \cdot  \vec{\uppercase\expandafter{\romannumeral2}}_{i j} $ is the squared norm of the second fundamental form, $\vec{H}=\sum_i \vec{\uppercase\expandafter{\romannumeral2}}_{i i} $ is the mean curvature vector field and $|\vec{H}|$ is the mean curvature of $\Sigma^{n}$.

In the following, we let 
\begin{align}
			f = |X|^2, \label{f}
		\end{align}
where $X$ is the position vector.  Then by a direct computation, we have
\begin{align}
    \nabla^{\Sigma} f = 2 X^{\top}, \  \ |\vec{H}|^2 = f - \frac{|\nabla^{\Sigma} f|^2}{4},  
\end{align}
\begin{align}
			\Delta^{\Sigma} f = 2n - 2|X^{\perp}|^2 = 2n-2|\vec{H}|^2 = 2n - 2f + \frac{1}{2} |\nabla^{\Sigma} f|^2, \label{lap1}
		\end{align}
		where $\top$ denotes the projection onto the tangent bundle of $\Sigma$.

As stated in the introduction, we first prove the following two lemmas to illustrate why we consider the pinching problem.
	\begin{lem}\label{lemmotived}
		Let $X:\Sigma^{2} \to \mathbb{R}^{3}$ be an $2-$dimensional compact self-shrinker with the  squared norm of the second fundamental satisfying
		\begin{align}
	|\uppercase\expandafter{\romannumeral2} |^2 < \frac{3}{2}, \label{lemcond} \end{align}
		then $X(\Sigma^{2}) = S^{2}(\sqrt{2})$. 
	\end{lem}
	\begin{proof}
		We consider 
		\begin{align}
			\int_{\Sigma^{2}} |H|^2 = \int_{\Sigma^{2}} (| \uppercase\expandafter{\romannumeral2} |^2 + 2K) = \int_{\Sigma^{2}} | \uppercase\expandafter{\romannumeral2} |^2 + 4 \pi \chi(\Sigma^{2}), \notag
		\end{align}
		where we use Gauss Bonnet Theorem. 
        
        On the other hand, from $\eqref{self-shrinker}$ and the following equation $\eqref{lap1}$, we have
		\begin{align}
			\int_{\Sigma^{2}}  |H|^2 = 2  Area(\Sigma^{2}). \notag 
		\end{align}
		Combining with equations above and $\eqref{lemcond}$ we have
		\begin{align}
			2  Area(\Sigma^{2}) < \frac{3}{2} Area(\Sigma^{2}) + 4 \pi \chi(\Sigma^{2}). \label{area esti}
		\end{align}
		This leads to $\chi(\Sigma^{2}) = 2-2g > 0$, hence $g = 0$ and  $\chi(\Sigma^{2}) = 2 > 0$.
        
        If $\Sigma^{2}$ is not embedded, then  by Li and Yau \cite{li1982new}, we have
		\begin{align}
			\int_{\Sigma^{2}}  |H|^2 \ge 32 \pi   ,  \  Area(\Sigma^{2}) \ge 16 \pi. \notag 
		\end{align}
		However, this contradicts  with $\eqref{area esti}$. 
        
        In  the case of embedding, it follows from the work of Brendle \cite{brendle2016embedded} that a compact embedded self-shrinker in $\mathbb{R}^3$ of genus zero   must be a round sphere. 
        
        Hence we  complete the proof of this lemma. 
	\end{proof}

        \begin{lem}
            Let $X:M^{n} \to \mathbb{R}^{n+1}(n \ge 2)$ be a compact embedded self-shrinker with constant squared norm of the second fundamental form, then 
            \begin{align}
                | \uppercase\expandafter{\romannumeral2} |^2 < \frac{3}{2},
            \end{align}
            then $X(M^{n}) = S^{n}(\sqrt{n})$. 
        \end{lem}
        \begin{proof}
       We perform the  computation directly by the definition \eqref{self-shrinker} (or we  directly refer to  the  equation at the bottom of page 428 in \cite{andrews2022extrinsic}, 
            \begin{align}
                (\Delta - \langle X, \nabla \cdot \ \rangle ) H = \Delta H - \langle X, \nabla H \rangle =  -(| \uppercase\expandafter{\romannumeral2} |^2 - 1) H.\notag 
            \end{align}
         
         If $H$ is a constant, by Aleksandrov theorem, we have $X(M) = S^{n}(\sqrt{n})$. 
            
 If $H$ is not a constant, by  our assumption that $| \uppercase\expandafter{\romannumeral2} |^2$ is constant, we know $H$ is a eigenfunction of the operator $-\Delta + \langle X, \nabla \cdot \ \rangle $  with the eigenvalue $| \uppercase\expandafter{\romannumeral2} |^2 - 1$.
            
            Hence        
            by the work of Ding and Xin (\cite{ding2013volume}, Theorem 1.3), we have
            \begin{align}
              | \uppercase\expandafter{\romannumeral2} |^2 - 1 \ge \frac{1}{2}.
            \end{align}
            This is contradiction to our assumption that $| \uppercase\expandafter{\romannumeral2} |^2 < \frac{3}{2}.$
        \end{proof}

        \begin{rem}
         This result improves Cheng and Wei's result \cite{cheng2015gap} in the case of compact embedding. 
        \end{rem}
         
	\section{Proof of Theorem \ref{maingradest}}\label{sec:prelims}
 
Before proving Theorem \ref{maingradest}, we need  the following several lemmas:
 
 \begin{lem} \label{meanmax}
		Let $X:\Sigma^{n} \to \mathbb{R}^{n+p} (n \ge 2)$ be an $n-$dimensional compact self-shrinker, 
		then we have 
		\begin{align}
			|X|_{\max} = |\vec{H}|_{\max} \ge \sqrt{n},
		\end{align}
		and if the equality holds, then $\Sigma^{n}$ is a minimal hypersurface in $S^{n+p-1}(\sqrt{n})$.
	\end{lem}
	\begin{proof}
		We choose a point $p \in \Sigma^{n}$ such that 
		\[
		|X|_{p} = |X|_{\max} .
		\] 
		Then by the first variation of the length and $\eqref{self-shrinker}$ we have
		\[
		|\vec{H}|(p) = |X|_{\max}.
		\]
		Combining with the fact that	$|\vec{H}| \le |X|,$
		we have 
		\[ |X|_{\max} = 	|\vec{H}|_{\max}. \]
		Integrating $\eqref{lap1}$ on $\Sigma^{n}$, we have 
		\[\int_{\Sigma^{n}} (2n - 2|X^{\perp}|^2) = \int_{\Sigma^{n}} (2n - 2|\vec{H}|^2)=0. \]
		Then 
		\[ |X|_{\max} = |\vec{H}|_{\max} \ge \sqrt{n}, \]
		If the equality holds, then we obtain $|H|^2 = |X|^2= n $. Thus, by  $\eqref{self-shrinker}$, we can conclude that 
		 \(\Delta^{\Sigma} X=-X \), i.e.
        \( M \) is a minimal submanifold in the sphere \( S^{n+p-1}(\sqrt{n}) \).
        
       Hence we prove this lemma. 
	\end{proof}
	
	In the following, we assume that $|X^{\top}|_{\max} >0$ and $ |X|_{\max}>\sqrt{n}$, otherwise Theorem \ref{maingradest}  follows immediately.
	\begin{lem} \label{ineqf}
		Let $X:\Sigma^{n} \to \mathbb{R}^{n+p}$ be an $n-$dimensional complete self-shrinker, then we have 
		\begin{align}
			\frac{1}{2} \Delta^{\Sigma} |\nabla^{\Sigma} f |^2 & \ge \frac{n}{n-1} |\nabla^{\Sigma} |\nabla^{\Sigma} f \|^2 + \frac{[2(n-f) + \frac{1}{2} |\nabla^{\Sigma} f |^2]^2}{n-1} \notag \\ & - \frac{1}{n-1}[2(n-f) + \frac{1}{2}|\nabla^{\Sigma} f|^2] \frac{\langle \nabla^{\Sigma} |\nabla^{\Sigma} f|^2, \nabla^{\Sigma} f \rangle }{| \nabla^{\Sigma} f |^2} \notag \\ & +  \frac{1}{2} \langle \nabla^{\Sigma} |\nabla^{\Sigma} f|^2, \nabla^{\Sigma} f \rangle - 2 | \nabla^{\Sigma} f |^2 \notag \\ & + [(1-\frac{2}{n}) \vec{H} \cdot \vec{\uppercase\expandafter{\romannumeral2}}_{11} - \frac{n-1}{n} |\vec{\uppercase\expandafter{\romannumeral2}}|^2 + \frac{|\vec{H}|^2}{n}] | \nabla^{\Sigma} f|^2. \label{ief}
		\end{align}
		where $|\nabla^{\Sigma} f| \neq 0$. 
	\end{lem}
	\begin{proof}
		 At the point $q^{\prime}$, if  $|\nabla^{\Sigma} f|(q^{\prime}) \neq 0$, then  we can  choose an orthonormal  frame        
 $ e_{1},e_{2},\cdots, e_{n}$ such that 
		\[ e_{1}(q^{\prime}) = \frac{\nabla^{\Sigma} f}{|\nabla^{\Sigma} f|}(q^{\prime}) \]
		Then 
		\[
		|\nabla^{\Sigma} |\nabla^{\Sigma} f|^2|^2 = 4 \sum^{n}_{j=1}(\sum^{n}_{i=1}f_{i}f_{ij})^2 =  4 |\nabla^{\Sigma} f|^2 (f^2_{11} + f^2_{12} + \cdots + f^2_{1n}).
		\]
		On the other hand,
		\[  |\nabla^{\Sigma} |\nabla^{\Sigma} f|^2|^2  = 4 |\nabla^{\Sigma} f|^2 |\nabla^{\Sigma} |\nabla^{\Sigma} f||^2.\]
		This implies
		\[  
		|\nabla^{\Sigma}|\nabla^{\Sigma} f||^2 = (f^2_{11} + f^2_{12} + \cdots + f^2_{1n}).
		\]
		Then
		\begin{align}
			|\nabla^{\Sigma} \nabla^{\Sigma} f|^2 & \ge f^2_{11} + f^2_{22} + \cdots + f^2_{nn} + 2(f^2_{12} + f^2_{13} + \cdots + f^2_{1n}) \notag \\ & \ge \frac{(\Delta^{\Sigma} f - f_{11})^2}{n-1} + f^2_{11} + 2(f^2_{12} + f^2_{13} + \cdots + f^2_{1n}) \notag \\ & \ge \frac{n}{n-1} (f^{2}_{11} + f^2_{12} + \cdots + f^2_{1n}) + \frac{|\Delta^{\Sigma} f|^2}{n-1} - \frac{2}{n-1} \Delta^{\Sigma} f \cdot f_{11}  \notag \\ & = \frac{n}{n-1} |\nabla^{\Sigma}|\nabla^{\Sigma} f||^2 + \frac{|\Delta^{\Sigma} f|^2}{n-1} - \frac{2}{n-1} \Delta^{\Sigma} f \cdot f_{11}. \label{hessest}
		\end{align}
		Since 
		\[
		f_{11} = \frac{1}{|\nabla^{\Sigma} f|^2} \langle \nabla^{\Sigma}_{ \nabla^{\Sigma} f} \nabla^{\Sigma} f, \nabla^{\Sigma} f \rangle = \frac{\langle \nabla^{\Sigma} f, \nabla^{\Sigma} |\nabla^{\Sigma} f|^2 \rangle}{2 |\nabla^{\Sigma} f|^2},
		\]
		we have
		\begin{align}
			|\nabla^{\Sigma} \nabla^{\Sigma} f|^2 \ge  \frac{n}{n-1} |\nabla^{\Sigma}|\nabla^{\Sigma} f||^2 + \frac{|\Delta^{\Sigma} f|^2}{n-1} - \frac{\Delta^{\Sigma} f \cdot \langle \nabla^{\Sigma} f, \nabla^{\Sigma} |\nabla^{\Sigma} f|^2 \rangle}{(n-1)|\nabla^{\Sigma} f|^2}. 
		\end{align}
		By Gauss equation $\eqref{eq:Ric}$, we have
		\begin{align}
			Ric^{\Sigma}_{11} = \vec{H} \cdot \vec{\uppercase\expandafter{\romannumeral2}}_{11} - \sum_{i=1}^{n} \vec{\uppercase\expandafter{\romannumeral2}}_{1i}^2. \notag 
		\end{align}
		Hence
		\begin{align}
			|\vec{\uppercase\expandafter{\romannumeral2}}|^2 & \ge \vec{\uppercase\expandafter{\romannumeral2}}^2_{11} + \vec{\uppercase\expandafter{\romannumeral2}}^2_{22} + \cdots + \vec{\uppercase\expandafter{\romannumeral2}}^2_{nn} + 2(\vec{\uppercase\expandafter{\romannumeral2}}^2_{12} + \vec{\uppercase\expandafter{\romannumeral2}}^2_{13} + \cdots + \vec{\uppercase\expandafter{\romannumeral2}}^2_{1n}) \notag \\ & \ge \frac{(\vec{H} - \vec{\uppercase\expandafter{\romannumeral2}}_{11})^2}{n-1} + \vec{\uppercase\expandafter{\romannumeral2}}^2_{11} + 2(\vec{\uppercase\expandafter{\romannumeral2}}^2_{12} + \vec{\uppercase\expandafter{\romannumeral2}}^2_{13} + \cdots + \vec{\uppercase\expandafter{\romannumeral2}}^2_{1n}) \notag \\ & \ge \frac{n}{n-1} (\vec{\uppercase\expandafter{\romannumeral2}}^{2}_{11} + \vec{\uppercase\expandafter{\romannumeral2}}^2_{12} + \cdots + \vec{\uppercase\expandafter{\romannumeral2}}^2_{1n}) + \frac{|\vec{H}|^2}{n-1} - \frac{2}{n-1} \vec{H} \cdot \vec{\uppercase\expandafter{\romannumeral2}}_{11}  \notag \\ & = \frac{n}{n-1} \sum_{i=1}^{n} \vec{\uppercase\expandafter{\romannumeral2}}_{1i}^2  + \frac{|\vec{H}|^2}{n-1} - \frac{2}{n-1} \vec{H} \cdot \vec{\uppercase\expandafter{\romannumeral2}}_{11} \notag \\ & =  -\frac{n}{n-1} Ric^{\Sigma}_{11} + \frac{n - 2}{n-1}  \vec{H} \cdot \vec{\uppercase\expandafter{\romannumeral2}}_{11}  +  \frac{|\vec{H}|^2}{n-1},\notag
		\end{align}
		that is,
		\begin{align}
			Ric^{\Sigma}_{11}  \ge (1 - \frac{2}{n})   \vec{H} \cdot \vec{\uppercase\expandafter{\romannumeral2}}_{11} + \frac{|\vec{H}|^2}{n} - \frac{n-1}{n} | \vec{\uppercase\expandafter{\romannumeral2}} |^2. \label{ricest2}
		\end{align}
		Then the conclusion  follows from $\eqref{lap1},\eqref{Lap2},\eqref{hessest}$ and $\eqref{ricest2}$.
	\end{proof}
	
	For any $\delta > 1$, let
	\begin{align}
		\tilde{f}_{\max} = \delta f_{\max} \in  (\delta \sqrt{n},\delta\vec{\uppercase\expandafter{\romannumeral2}} |_{\max}^2 ]\notag ,
	\end{align}
	and 
	\begin{align}
		g = \frac{|\nabla^{\Sigma} f|^2}{\tilde{f}_{\max} - f}.  \label{g}
	\end{align}
	Since $\Sigma^{n}$ is compact, we can choose the point $x_{0}$ where it attains its maximum, and we denote this maximum by $a$ i.e.
	\begin{align}
		a = g_{\max} = g(x_{0})>0.
	\end{align} 
	Then at  this point, we have 
	\begin{lem}\label{ineqquad}
		At  the maximum point of $g$ above, we have
		\begin{align}
			0  \geq &   [\frac{1}{4}(\tilde{f}_{\max} -f)^2 + \frac{1}{2} (\tilde{f}_{\max} -f) + \frac{n^2}{4}  ]a^2 \notag \\ + & \{ (n+1)(\tilde{f}_{\max} -f)^2 + n(n+1)(n - \tilde{f}_{\max}) \notag \\ + & [2(n^2+1) -(n+1)\tilde{f}_{\max} -  (n-1)^2 |\vec{\uppercase\expandafter{\romannumeral2}}|^2 ](\tilde{f}_{\max} -f)  \} a \notag \\ + & 4n (f-n)^2. \label{eqquada}
		\end{align}
	\end{lem}
	\begin{proof}
		First we note that
		\begin{align}
			\frac{\langle \nabla^{\Sigma} |\nabla^{\Sigma} f|^2, \nabla^{\Sigma} f \rangle}{|\nabla^{\Sigma} f|^2} & = 2f_{11} = 4 \langle \nabla^{\Sigma}_{e_{1}} X^{\top}, e_{1}  \rangle \notag \\ &  = 4 \langle \nabla^{R^{n+p}}_{e_{1}} X- \nabla^{R^{n+p}}_{e_{1}} X^{\perp} , e_{1}  \rangle  \notag \\ &   =  4( 1 - \vec{H} \cdot \vec{\uppercase\expandafter{\romannumeral2}}_{11}).  \notag 
		\end{align}
		Thus we have
		\begin{align}
			\vec{H} \cdot \vec{\uppercase\expandafter{\romannumeral2}}_{11} = - \frac{\langle \nabla^{\Sigma} |\nabla^{\Sigma} f|^2, \nabla^{\Sigma} f \rangle }{4|\nabla^{\Sigma} f |^2} + 1. \label{Hll1}
		\end{align}
		Now we compute
		\begin{align}
			\frac{1}{2} \Delta^{\Sigma} g = \frac{1}{2} \Delta^{\Sigma}( \frac{|\nabla^{\Sigma} f |^2}{\tilde{f}_{\max} - f} ) & = \frac{1}{2} \frac{\Delta^{\Sigma} |\nabla^{\Sigma} f |^2}{\tilde{f}_{\max} - f}  + \frac{\langle \nabla^{\Sigma} |\nabla^{\Sigma} f|^2, \nabla^{\Sigma} f \rangle}{(\tilde{f}_{\max} - f)^2} \notag \\ & + \frac{1}{2} \frac{1}{(\tilde{f}_{\max} -f)^2 }[\Delta^{\Sigma} f + 2 \frac{|\nabla^{\Sigma} f |^2}{\tilde{f}_{\max} -f}] | \nabla^{\Sigma} f|^2  \notag \\ & = \frac{1}{2} \frac{\Delta^{\Sigma} |\nabla^{\Sigma} f |^2}{\tilde{f}_{\max} - f}  + \frac{\langle \nabla^{\Sigma} |\nabla^{\Sigma} f|^2, \nabla^{\Sigma} f \rangle}{(\tilde{f}_{\max} - f)^2} \notag \\ & + \frac{1}{2} \frac{1}{(\tilde{f}_{\max} -f)^2 }[2(n-f) + \frac{1}{2} | \nabla^{\Sigma} f|^2 + 2 \frac{|\nabla^{\Sigma} f |^2}{\tilde{f}_{\max} -f}] |\nabla^{\Sigma} f|^2.  \label{lapg}
		\end{align}
		At the maximum point $x_{0}$, we have
		\begin{align}
			\nabla^{\Sigma} g (x_{0}) = 0, \  \ \Delta^{\Sigma} g(x_{0}) \le 0. \notag 
		\end{align}
		Then we have 
		\begin{align}
			|\nabla^{\Sigma} f|^2(x_{0}) = \frac{|\nabla^{\Sigma} f |^2}{\tilde{f}_{\max} -f} \cdot (\tilde{f}_{\max} -f) (x_{0}) = a(\tilde{f}_{\max} -f) (x_{0}) 
		\end{align}
		and
       \begin{align}
			|\vec{H}|^2(x_{0}) = f(x_{0}) - \frac{| \nabla^{\Sigma} f|^2 }{4}(x_{0}) = f(x_{0}) - \frac{a(\tilde{f}_{\max} - f)(x_{0})}{4},
		\end{align}
        
        From the fact that $\nabla^{\Sigma} g (x_{0}) = 0$, we have
		\begin{align}
			\nabla^{\Sigma}  |\nabla^{\Sigma} f |^2 (x_{0}) = -\frac{\nabla^{\Sigma} f}{\tilde{f}_{\max} -f}  |\nabla^{\Sigma} f |^2 (x_{0}),
		\end{align}
		which implies
		
		\begin{align}
			| \nabla^{\Sigma} |\nabla^{\Sigma} f||^2(x_{0}) = \frac{|\nabla^{\Sigma} f |^4 }{4(\tilde{f}_{\max} -f)^2} (x_{0}) = \frac{a^2}{4},
		\end{align}
		and
		\begin{align}
			\langle \nabla^{\Sigma} |\nabla^{\Sigma} f|^2, \nabla^{\Sigma} f \rangle (x_{0}) = -\frac{| \nabla^{\Sigma} f |^4 }{\tilde{f}_{\max} - f} (x_{0}) = -a^2 (\tilde{f}_{\max} - f) (x_{0}).
		\end{align}
	 Putting  these equations above into $\eqref{lapg}$ and by $\eqref{Hll1}$,Lemma \ref{ineqf}, then  we have the inequality $\eqref{eqquada}.$ 
		
        Hence we prove this  lemma. 
	\end{proof}

Now, we prove Theorem \ref{maingradest}.

\begin{proof}[Proof of Theorem \ref{maingradest}]

First, we note that the following inequality:
\begin{align}
\frac{1}{4} a^2+4(f-n)^2 \geq 2 \mid f-n \mid a.\notag
\end{align}
At the maximum point of $g$, we put it into Lemma \ref{ineqquad} to obtain
\begin{align}
0  \geq &   [\frac{1}{4}(\tilde{f}_{\max} -f)^2 + \frac{1}{2} (\tilde{f}_{\max} -f) + \frac{n(n-1)}{4}  ] a \notag \\ +&  (n+1)(\tilde{f}_{\max} -f)^2 + n(n+1)(n - \tilde{f}_{\max}) +2  n \mid f-n \mid \notag \\
+ & [2(n^2+1) -(n+1)\tilde{f}_{\max} -  (n-1)^2 |\vec{\uppercase\expandafter{\romannumeral2}}|_{\max} ^2 ](\tilde{f}_{\max} -f).\label{first ineq}
\end{align}
$1^{\circ}$ If $f\geq n$ at this point, then by \eqref{first ineq}, we have
\begin{align}
0  \geq &   [\frac{1}{4}(\tilde{f}_{\max} -f)^2 + \frac{1}{2} (\tilde{f}_{\max} -f) + \frac{n(n-1)}{4}  ] a \notag \\ +&  (n+1)(\tilde{f}_{\max} -f)^2 + n(n-1 )(n - \tilde{f}_{\max}) \notag \\
+ & [2(n^2-n+1) -(n+1)\tilde{f}_{\max} -  (n-1)^2 |\vec{\uppercase\expandafter{\romannumeral2}}|_{\max} ^2 ](\tilde{f}_{\max} -f)\notag \\
>& [\frac{n(n-1)}{4}  +\frac{1}{2} (\tilde{f}_{\max} -f)]a+ n(n-1 )(n - \tilde{f}_{\max})  \notag \\
+ & [2(n^2- n+1) -(n+1)\tilde{f}_{\max} -  (n-1)^2 |\vec{\uppercase\expandafter{\romannumeral2}}|_{\max} ^2 ](\tilde{f}_{\max} -f),\notag 
\end{align}
that is, 
\begin{align}
 a<&\frac{[-2(n^2- n+1) +(n+1)\tilde{f}_{\max} +(n-1)^2 |\vec{\uppercase\expandafter{\romannumeral2}}|_{\max} ^2](\tilde{f}_{\max} -f) +n(n-1 )( \tilde{f}_{\max}-n)}{\frac{n(n-1)}{4}  +\frac{1}{2} (\tilde{f}_{\max} -f)}\notag \\
 \leq &4(\tilde{f}_{\max}-n)\max\{1,\frac{-n^2+n-2+(n+1)\tilde{f}_{\max} +(n-1)^2 |\vec{\uppercase\expandafter{\romannumeral2}}|_{\max} ^2}{n(n- 1)  +2(\tilde{f}_{\max} -n)} \}\notag \\
 =& 4(\tilde{f}_{\max}-n) \frac{-n^2+n-2+(n+1)\tilde{f}_{\max} +(n-1)^2 |\vec{\uppercase\expandafter{\romannumeral2}}|_{\max} ^2}{n(n-3)  +2\tilde{f}_{\max}}.\label{frist result}
 \end{align}
$2^{\circ}$ If $f< n $ at this point, then by \eqref{first ineq}, we have 
\begin{align}
0  \geq &   [\frac{1}{4}(\tilde{f}_{\max} -f)^2 + \frac{1}{2} (\tilde{f}_{\max} -f) + \frac{n(n-1)}{4}  ] a \notag \\ +&  (n+1)(\tilde{f}_{\max} -f)^2 + n(n+3)(n - \tilde{f}_{\max}) \notag \\
+ & [2(n^2+ n+1) -(n+1)\tilde{f}_{\max} -  (n-1)^2 |\vec{\uppercase\expandafter{\romannumeral2}}|_{\max} ^2 ](\tilde{f}_{\max} -f) \notag \\
>& [\frac{1}{4}(\tilde{f}_{\max} -n)^2 + \frac{1}{2} (\tilde{f}_{\max} -n) + \frac{n(n-1)}{4}  ] a \notag \\ +&  (n+1)(\tilde{f}_{\max} -f)^2 + n(n+3 )(n - \tilde{f}_{\max}) \notag \\
+ & [2(n^2+n+1) -(n+1)\tilde{f}_{\max} -  (n-1)^2 |\vec{\uppercase\expandafter{\romannumeral2}}|_{\max} ^2 ](\tilde{f}_{\max} -f), \notag 
\end{align}
that is,
\begin{align}
a< &\frac{[-2(n^2+ n+1) +(n+1)\tilde{f}_{\max} +(n-1)^2 |\vec{\uppercase\expandafter{\romannumeral2}}|_{\max} ^2 ](\tilde{f}_{\max} -f)}{\frac{1}{4}(\tilde{f}_{\max} -n)^2 + \frac{1}{2} (\tilde{f}_{\max} -n) + \frac{n(n-1)}{4}}
\notag \\
+&\frac{-(n+1)(\tilde{f}_{\max} -f)^2+n(n+3 )(\tilde{f}_{\max}-n)}{\frac{1}{4}(\tilde{f}_{\max} -n)^2 + \frac{1}{2} (\tilde{f}_{\max} -n) + \frac{n(n-1)}{4}}.\notag
\end{align}
We  find  that when 
$$\frac{-2(n^2+ n+1) +(n+1)\tilde{f}_{\max} +(n-1)^2 |\vec{\uppercase\expandafter{\romannumeral2}}|_{\max} ^2}{ 2(n+1)   }\leq \tilde{f}_{\max}-n,$$
i.e.  
$$(n-1)^2|\vec{\uppercase\expandafter{\romannumeral2}}|_{\max} ^2 \leq (n+1)\tilde{f}_{\max} +2,$$
we can obtain
\begin{align}
a< &\frac{[-2(n^2+ n+1) +(n+1)\tilde{f}_{\max} +(n-1)^2 |\vec{\uppercase\expandafter{\romannumeral2}}|_{\max} ^2 ](\tilde{f}_{\max} -n)}{\frac{1}{4}(\tilde{f}_{\max} -n)^2 + \frac{1}{2} (\tilde{f}_{\max} -n) + \frac{n(n-1)}{4}}
\notag \\
+&\frac{-(n+1)(\tilde{f}_{\max} -n)^2+n(n+3)(\tilde{f}_{\max}-n)}{\frac{1}{4}(\tilde{f}_{\max} -n)^2 + \frac{1}{2} (\tilde{f}_{\max} -n) + \frac{n(n-1)}{4}}\notag \\
=& 4(n-1)(\tilde{f}_{\max}-n)\frac{(n-1)(|\vec{\uppercase\expandafter{\romannumeral2}}|_{\max} ^2-1)+(n+1)}{(\tilde{f}_{\max} -n)^2 + 2 (\tilde{f}_{\max} -n) + n(n-1)}\notag \\
<& 4(\tilde{f}_{\max}-n) \frac{-n^2+n-2+(n+1)\tilde{f}_{\max} +(n-1)^2 |\vec{\uppercase\expandafter{\romannumeral2}}|_{\max} ^2}{n(n-3)  +2\tilde{f}_{\max} }.\label{two result}
\end{align}
Considering the above two cases and combing \eqref{frist result} with \eqref{two result}, we can conclude that  when 
$$(n-1)^2|\vec{\uppercase\expandafter{\romannumeral2}}|_{\max} ^2 \leq (n+1)f_{\max} +2,$$
\begin{align}
 a<4(\tilde{f}_{\max}-n) \frac{-n^2+n-2+(n+1)\tilde{f}_{\max} +(n-1)^2 |\vec{\uppercase\expandafter{\romannumeral2}}|_{\max} ^2}{n(n-3)  +2\tilde{f}_{\max} }.\notag
\end{align}
Let $\delta \to 1$, we immediately see that   for any point in $\Sigma$,
\begin{align}
|\nabla^{\Sigma} f|^2 \leq 4(f_{\max}-n) \frac{-n^2+n-2+(n+1)f_{\max} +(n-1)^2 |\vec{\uppercase\expandafter{\romannumeral2}}|_{\max} ^2}{n(n-3)  +2f_{\max} } (f_{\max}-f).\label{final result}
\end{align}
This completes the proof of Theorem \ref{maingradest}.

\end{proof}

	\section{Proof of Theorem \ref{main} } \label{mainproof}
	Now, we prove Theorem \ref{main}.  
 
 \begin{proof}
    
      First, we assume that  the squared norm of the second fundamental satisfies the assumption \eqref{hypothesis} of Theorem \ref{maingradest}.
      
      We  choose a point $p$ such that $|X|^2_{\max} = |X|^2(p) $ and  choose  $q_0$  such that  $d^{\Sigma}(p,q_0) = r_0=\max\limits_{q\in \Sigma} d^{\Sigma}(p,q)$.
      
      Then we can choose a minimizing geodesic $\gamma$  with   unit speed such that $\gamma(0)=p,\gamma(r_0)=q_0$.

      Hence when  we restrict   $|X|^2 $ and $|\vec{H}|^{2}$ on $\gamma (t) $ and by Theorem \ref{maingradest}, 
		  for any $0\leq t\leq r_0$, we have
        \begin{align}
			 |X|^2_{\max} -|X|^2 \leq &B^2(| X|^2_{\max}-n) t^2, \label{|X|}  \\
             |\vec{H}|^{2} = |X|^2 - |X^\top|^2 \geq &|X|^2_{\max}-[B^2(| X|^2_{\max}-n)+1]( |X|^2_{\max} -|X|^2) \notag \\
             \geq & |X|^2_{\max}-B^2[B^2(| X|^2_{\max}-n)+1](| X|^2_{\max}-n) t^2.\label{|H|}
             \end{align}
		If we choose $e_1=\gamma'$ and $e_2,\cdots ,e_n $ such that $e_2,\cdots ,e_n \perp \gamma'$  and are parallel along $\gamma$, then we  have same conclusion as \eqref{ricest2}:
\begin{align}
			Ric^{\Sigma}(\gamma',\gamma')  \geq  \frac{n-2}{n}   \vec{H} \cdot \vec{\uppercase\expandafter{\romannumeral2}}_{11} + \frac{|\vec{H}|^2}{n} - \frac{n-1}{n} | \vec{\uppercase\expandafter{\romannumeral2}} |^2.\label{Ricci}
            \end{align}
And for any $0 <l \leq r_0$, we  let $$V_{j}(t) = \sin(\frac{\pi}{l} t)e_{j}(t), \qquad j = 2,\cdots,n. $$ It is clear that $V_{j}$ generates a proper variation of $\gamma$ in $[0,l]$, whose energy is  denoted by $E_{j}$.  
     
      Then by the second variational formula of the energy and \eqref{Ricci},  for any $0<l \leq r_0$,
        \begin{align}
           0 \le  \frac{1}{2} \sum_{j=1}^{n}  E^{\prime \prime}_{j}(0) =&\frac{(n-1)\pi^2}{l^2} \int_{0}^{l} \cos^2 (\frac{\pi}{l}t) dt - \int_{0}^{l} \sin^2(\frac{\pi}{l}t) Ric^{\Sigma}(\gamma',\gamma')dt \notag \\
           \leq &\frac{(n-1)^2 \pi^2 }{2l}-\int_{0}^{l} \sin^2(\frac{\pi}{l}t) (\frac{n-2}{n}   \vec{H} \cdot \vec{\uppercase\expandafter{\romannumeral2}}_{11} + \frac{|\vec{H}|^2}{n} - \frac{n-1}{n} | \vec{\uppercase\expandafter{\romannumeral2}} |^2) dt,
           \label{variation 1} 
        \end{align}
where we use the equality
$$\int_{0}^{l} \sin^2(\frac{\pi}{l}t)dt=\int_{0}^{l} \cos^2(\frac{\pi}{l}t)dt=\frac{l}{2}.$$

Then we plug the equality 
 $$\frac{1}{2}(|X|^2_{\max}-|X|^2)''=\vec{H} \cdot \vec{\uppercase\expandafter{\romannumeral2}}_{11}-1$$
 and the inequality\eqref{|H|}into \eqref{variation 1} to obtain:
 \begin{align}
           0\leq & \frac{(n-1)^2 \pi^2 }{2l}-\int_{0}^{l} \sin^2(\frac{\pi}{l}t) \{\frac{n-2}{n} (\frac{1}{2}(|X|^2_{\max}-|X|^2)''+1) +\frac{1}{n}|X|^2_{\max} \notag \\
           -&\frac{1}{n}[B^2(| X|^2_{\max}-n)+1]( |X|^2_{\max} -|X|^2) - \frac{n-1}{n} | \vec{\uppercase\expandafter{\romannumeral2}} |^2_{\max}\} dt\notag\\
           =&\frac{l}{2}[\frac{(n-1)\pi^2}{l^2}-\frac{n-2}{n }-\frac{|X|^2_{\max}}{n}+\frac{n-1}{n}|\vec{\uppercase\expandafter{\romannumeral2}} |^2_{\max}]\notag \\
           -&\int_{0}^{l}\{\frac{(n-2)\pi^2}{n l^2}\cos(\frac{2\pi}{l}t)-\frac{1}{n}[B^2(| X|^2_{\max}-n)+1]\sin^2(\frac{\pi}{l}t)\}(|X|^2_{\max}-|X|^2)dt\notag \\
           =&\frac{l}{2}[\frac{(n-1)\pi^2}{l^2}-\frac{n-2}{n }-\frac{|X|^2_{\max}}{n}+\frac{n-1}{n}|\vec{\uppercase\expandafter{\romannumeral2}} |^2_{\max}] +\int_{0}^{l} \{\frac{1}{2n}[B^2(| X|^2_{\max}-n)+1]\notag \\       &-[\frac{(n-2)\pi^2}{n l^2}+\frac{1}{2n}[B^2(| X|^2_{\max}-n)+1]] \notag\\
           &\cos(\frac{2\pi}{l}t)\}(|X|^2_{\max}-|X|^2)dt.\label{varation2} 
           \end{align}
We can see that  when $l_0 \leq t\leq l-l_0$,
$$\frac{1}{2n}[B^2(| X|^2_{\max}-n)+1]-[\frac{(n-2)\pi^2}{n l^2}+\frac{1}{2n}[B^2(| X|^2_{\max}-n)+1]] \cos(\frac{2\pi}{l}t) \geq 0,$$
where $0<l_0<\frac{l}{4}$ and 
$$\cos(\frac{2\pi}{l}l_0)=\frac{B^2(| X|^2_{\max}-n)+1}{B^2(| X|^2_{\max}-n)+1+ \frac{2(n-2)\pi^2}{l^2}}.$$
We also note that 
\begin{align}
\int_{s}^{l-s} t^2\cos (\frac{2\pi}{l}t) dt=\frac{l}{2\pi}\sin( \frac{2\pi}{l}s)[ \frac{l^2}{\pi^2}-s^2-(l-s)^2]+\frac{l^2}{2\pi^2} \cos (\frac{2\pi}{l}s)(l-2s),\notag 
\end{align}
and 
\begin{align}
    \frac{d}{ds}\int_{s}^{l-s} t^2\cos (\frac{2\pi}{l}t) dt<0, \forall 0<s<\frac{l}{4}.\notag
    \end{align}
That is,
\begin{align}
\int_{s}^{l-s} t^2\cos (\frac{2\pi}{l}t) dt >\int_{\frac{3l}{4}}^{\frac{l}{4}} t^2\cos (\frac{2\pi}{l}t) dt=-\frac{l^3}{2\pi}( -\frac{1}{\pi^2}+\frac{1}{16}+\frac{9}{16}), \forall 0<s<\frac{l}{4}.\notag
\end{align}
We substitute the above results and  the inequality \eqref{|X|} into \eqref{varation2} and multiply both sides of \eqref{varation2}  by $\frac{2}{l}$ to obtain 
\begin{align}
    0< &\frac{(n-1)\pi^2}{l^2}-\frac{n-2}{n }-\frac{|X|^2_{\max}}{n}+\frac{n-1}{n}|\vec{\uppercase\expandafter{\romannumeral2}} |^2_{\max}\notag\\
    +&B^2(|X|^2_{\max}-n)l^2\{\frac{1}{3n}(B^2(| X|^2_{\max}-n)+1) \notag\\
    +&\frac{1}{\pi}( -\frac{1}{\pi^2}+\frac{1}{16}+\frac{9}{16}) 
    [\frac{(n-2)\pi^2}{n l^2}+\frac{1}{2n}(B^2(| X|^2_{\max}-n)+1)]\} \notag\\
   =&\frac{(n-1)\pi^2}{l^2}+(\frac{1}{3}+\frac{5}{16\pi}-\frac{1}{2\pi^3})\frac{1}{n}B^2(|X|^2_{\max}-n)[B^2(| X|^2_{\max}-n)+1]l^2  \notag\\
   +& (\frac{5\pi}{8}-\frac{1}{\pi})\frac{n-2}{n}B^2(|X|^2_{\max}-n)-\frac{n-2}{n }-\frac{|X|^2_{\max}}{n}+\frac{n-1}{n}|\vec{\uppercase\expandafter{\romannumeral2}} |^2_{\max}.\label{variation3}
   \end{align}

Here let's consider the the lower  bound of $ |\vec{H}|^{2}$ and the upper bound of $|\vec{\uppercase\expandafter{\romannumeral2}} |^2-\frac{ |\vec{H}|^{2}}{n}$ first. 

For any $q\in \Sigma$, by \eqref{|H|}, we have 
 \begin{align}
|\vec{H}|^{2} \geq & |X|^2_{\max}-B^2[B^2(| X|^2_{\max}-n)+1](| X|^2_{\max}-n) r_0^2,\notag\\
|\vec{\uppercase\expandafter{\romannumeral2}} |^2-\frac{ |\vec{H}|^{2}}{n} \leq & |\vec{\uppercase\expandafter{\romannumeral2}} |^2_{\max}-\frac{|X|^2_{\max}}{n}+B^2[B^2(| X|^2_{\max}-n)+1](\frac{| X|^2_{\max}}{n}-1) r_0^2.\label{global}
\end{align}
 
 Since if the codimension is 1 and $|\vec{H}|^{2} >0$ for any point, then by 
Huisken's rigidity theorm on mean convex self-shrinkers\cite{huisken1990asymptotic}, we immediately have $X(\Sigma^{n}) = S^{n}(\sqrt{n})$.

If the codiemsion is more than 1 and  $``|\vec{H}|^{2} >0,|\vec{\uppercase\expandafter{\romannumeral2}} |^2-\frac{ |\vec{H}|^{2}}{n}\leq \frac{1}{2}^"$, then by Theorem 1.6 of  \cite{cao2014pinching},  we immediately have $X(\Sigma^{n}) = S^{n}(\sqrt{n})$. 

 So in the following step, to complete the proof of the main theorem, we need to restrict $|\vec{\uppercase\expandafter{\romannumeral2}} |^2_{\max}$ to ensure it.

From $\eqref{global}$, we can see that if 
\begin{align}
B^2[B^2(| X|^2_{\max}-n)+1](| X|^2_{\max}-n) r_0^2 <&\min\{|X|^2_{\max},\frac{n}{2}-(n|\vec{\uppercase\expandafter{\romannumeral2}} |_{\max}^2-|X|_{\max}^{2})\}\notag\\
=&\frac{n}{2}-(n|\vec{\uppercase\expandafter{\romannumeral2}} |_{\max}^2-|X|_{\max}^{2}),\label{condition}
\end{align}
the above conditions on $``|\vec{H}|^{2},|\vec{\uppercase\expandafter{\romannumeral2}} |^2-\frac{ |\vec{H}|^{2}}{n} \;"$ are satisfied.

To derive \eqref{condition}, combining  \eqref{variation3}, \eqref{condition} and the assumption\eqref{hypothesis}, and noting that \eqref{variation3} holds for any $0<l\leq r_0$, it suffices to ensure
\begin{align}
|\vec{\uppercase\expandafter{\romannumeral2}} |_{\max}^2 <\min\{\frac{|X|_{\max}^{2}}{n}+\frac{1}{2},\frac{n+1}{(n-1)^2}|X|_{\max}^{2}+\frac{2}{(n-1)^2}\} \label{condition''}
\end{align}
and 
\begin{align}
0 \geq & B^2[B^2(| X|^2_{\max}-n)+1](| X|^2_{\max}-n) \frac{(n-1)\pi^2}{\frac{n}{2}-(n|\vec{\uppercase\expandafter{\romannumeral2}} |_{\max}^2-|X|_{\max}^{2})}\notag\\
+& (\frac{1}{3}+\frac{5}{16\pi}-\frac{1}{2\pi^3})(\frac{1}{2}+\frac{|X|_{\max}^{2}}{n}-|\vec{\uppercase\expandafter{\romannumeral2}} |_{\max}^2 )+(\frac{5\pi}{8}-\frac{1}{\pi})\frac{n-2}{n}B^2(|X|^2_{\max}-n)\notag\\
-&\frac{n-2}{n }-\frac{|X|^2_{\max}}{n}+\frac{n-1}{n}|\vec{\uppercase\expandafter{\romannumeral2}} |^2_{\max}\notag\\
=&B^2[B^2(| X|^2_{\max}-n)+1](| X|^2_{\max}-n) \frac{(n-1)\pi^2}{\frac{n}{2}-n(|\vec{\uppercase\expandafter{\romannumeral2}} |_{\max}^2-1)+(|X|_{\max}^{2}-n)}\notag\\
+&[(-\frac{2}{3}+\frac{5}{16\pi} -\frac{1}{2\pi^3})\frac{1}{n}+(\frac{5\pi}{8}-\frac{1}{\pi})\frac{n-2}{n}B^2](|X|^2_{\max}-n)-\frac{5}{6}+\frac{5}{32\pi}-\frac{1}{4\pi^3}+\frac{1}{n}\notag\\
+&(\frac{2}{3}-\frac{5}{16\pi}+\frac{1}{2\pi^3}-\frac{1}{n})(|\vec{\uppercase\expandafter{\romannumeral2}} |^2_{\max}-1).\label{condition'}
\end{align}
Since 
\begin{align} 0 \leq |X|^2_{\max}-n =|\vec{H}|^2_{\max}-n \leq n (|\vec{\uppercase\expandafter{\romannumeral2}} |_{\max}^2-1) \label{basic}\end{align}
by Lemma \ref{meanmax},
and by the definition of $B:$ $$B^2=\frac{-n^2+n-2+(n+1)|X|^2_{\max} +(n-1)^2 |\vec{\uppercase\expandafter{\romannumeral2}}|_{\max}^2}{n(n-3)  +2|X|^2_{\max} } >\frac{n+1}{n},$$
to ensure the inequalities \eqref{condition''} and \eqref{condition'}, it suffices to ensure 
\begin{align}
   |\vec{\uppercase\expandafter{\romannumeral2}} |_{\max}^2 -1<&\min\{ \frac{1}{2},\frac{|X|_{\max}^{2}}{n}+\frac{1}{2},\frac{n+1}{(n-1)^2}|X|_{\max}^{2}+\frac{2}{(n-1)^2}-1\} \notag \\
   =&\min \{\frac{1}{2},\frac{n+1}{(n-1)^2}|X|_{\max}^{2}+\frac{2}{(n-1)^2}-1\}
   \label{condition'''}
   \end{align}
and 
\begin{align}
0\geq&(|\vec{\uppercase\expandafter{\romannumeral2}} |_{\max}^2 -1)\{2(n-1)\pi^2\tilde{B}^2[n\tilde{B}^2(|\vec{\uppercase\expandafter{\romannumeral2}} |_{\max}^2 -1)+1]\notag\\
+& (\frac{5\pi}{8}-\frac{1}{\pi})(n-2)\tilde{B}^2-\frac{1}{n}\}-\frac{5}{6}+\frac{5}{32\pi}-\frac{1}{4\pi^3}+\frac{1}{n} \notag\\
\label{condition''''}
\end{align}
where under the inequality \eqref{basic},
$$\tilde{B}^2=\frac{(2n^2-n+1)(|\vec{\uppercase\expandafter{\romannumeral2}} |_{\max}^2 -1)+n^2-1}{n[n-1+2(|\vec{\uppercase\expandafter{\romannumeral2}} |_{\max}^2 -1) ]} \geq B^2.$$ 
In fact, we can find   that the function \eqref{condition'} (fix $(|\vec{\uppercase\expandafter{\romannumeral2}} |_{\max}^2 -1)$) and   replace $B^2$ by $\tilde{B}^2$) must increase or decrease first and then increase on $(|X|_{\max}^{2}-n )$ by simple differentiation, or we can observe that this function is actually the type :$\alpha (x+\beta) + \frac{\Gamma_1}{ x+\beta}+\Gamma_2 (\alpha,\beta>0)$, so it attains  its maximum at the endpoint value. And we can clearly see that  the function attains a greater value at $n(|\vec{\uppercase\expandafter{\romannumeral2}} |_{\max}^2 -1)$ than at $0$. Thus we substitute   $n(|\vec{\uppercase\expandafter{\romannumeral2}} |_{\max}^2 -1)$ into  the function \eqref{condition'} to obtain  the condition \eqref{condition''''}.

We solve the quadratic inequality \eqref{condition''''} on $\tilde{B}^2  (|\vec{\uppercase\expandafter{\romannumeral2}} |_{\max}^2 -1)$ to derive
\begin{align}
    \tilde{B}^2(|\vec{\uppercase\expandafter{\romannumeral2}} |_{\max}^2 -1) &\leq \frac{c+\frac{1}{n}(|\vec{\uppercase\expandafter{\romannumeral2}} |_{\max}^2 -1)}{\pi[b+\sqrt{b^2+2n(n-1) c +2(n-1)(|\vec{\uppercase\expandafter{\romannumeral2}} |_{\max}^2 -1)}]}.\label{final inequality} \\
    &\left(> \frac{0.8-\frac{1}{n}+\frac{1}{n}(|\vec{\uppercase\expandafter{\romannumeral2}} |_{\max}^2 -1)}{\pi [ 3.5n-3.6+\sqrt{11n^2-24.8n+14.5+2(n-1)(|\vec{\uppercase\expandafter{\romannumeral2}} |_{\max}^2 -1)}  ]   }\notag \right.\\
 &\left. > \frac{0.8-\frac{1}{n}+\frac{1}{n}(|\vec{\uppercase\expandafter{\romannumeral2}} |_{\max}^2 -1 )}{\pi      [ 3.5n-3.6+ 3.4 (n-1) + \sqrt{2(n-1)(|\vec{\uppercase\expandafter{\romannumeral2} }|_{\max}^2 -1)}] }\notag \right.\\
  &\left.> \frac{0.8-\frac{1}{n} }{\pi\sqrt{n-1}    ( 6.9 \sqrt{n-1} + \sqrt{2(|\vec{\uppercase\expandafter{\romannumeral2} }|_{\max}^2 -1)})}.\notag
\right)
\end{align}

where $$b=(n-1)\pi +(\frac{5}{16}-\frac{1}{2\pi^2})(n-2),c=\frac{5}{6}-\frac{5}{32\pi}+\frac{1}{4\pi^3}-\frac{1}{n}.$$
We can find that when 
$$|\vec{\uppercase\expandafter{\romannumeral2}} |_{\max}^2 -1\leq \frac{1}{10\pi (n+2)},$$
 we have 
 \begin{align}
\frac{0.8-\frac{1}{n} }{\pi\sqrt{n-1}    ( 6.9 \sqrt{n-1} + \sqrt{2(|\vec{\uppercase\expandafter{\romannumeral2} }|_{\max}^2 -1)}) } \geq  & \frac{0.8-\frac{1}{n} }{\pi\sqrt{n-1}    [ 6.9 \sqrt{n-1} + \frac{1}{\sqrt{5\pi (n+2)}}] } \notag \\
>& \frac{0.8n-1 }{  7.1 \pi n(n-1) }\notag
\end{align}
and 
\begin{align}
 \tilde{B}^2(|\vec{\uppercase\expandafter{\romannumeral2}} |_{\max}^2 -1) <&
 \frac{1}{10\pi(n+2)} \times \frac{[(2n^2-n+1)\times \frac{1}{10\pi(n+2)} +n^2-1]}{n(n-1)}\notag\\
 <&\frac{1}{10\pi(n+2)} \times \frac{\frac{n-1}{5\pi}+ n^2-1}{n(n-1)}\notag\\
 =&\frac{n+1+\frac{1}{5\pi}}{10\pi n(n+2)  }.\notag
 \end{align}
By direct comparison, we can clearly  see that 
$$\frac{n+1+\frac{1}{5\pi}}{10\pi n(n+2)  }<\frac{0.8n-1 }{  7.1 \pi n(n-1) },\forall n \geq 2.$$
So the above inequality \eqref{final inequality} holds, that is, the conditions  \eqref{condition''''} also holds. And we can observe that at this time, the condition \eqref{condition'''} clearly holds under this pinching condition on the squared norm of  the second fundamental form.

This completes the   proof of the main theorem $\ref{main}$.
\end{proof}

 \vspace{1cm}

 \textbf{\Large Declarations }

\vspace{0.2cm}
\textbf{Conflict and interest:} On behalf of all authors,the corresponding author states 

that there is no conflict of interest.

	\bibliographystyle{siam}
	\bibliography{ref}

\begin{thebibliography}{10}

\bibitem{abresch1986normalized}
{\sc U.~Abresch and J.~Langer}, {\em The normalized curve shortening flow and homothetic solutions}, Journal of Differential Geometry, 23 (1986), pp.~175--196.

\bibitem{alencar2022hopf}
{\sc H.~Alencar, G.~Silva~Neto, and D.~Zhou}, {\em Hopf-type theorem for self-shrinkers}, Journal f{\"u}r die reine und angewandte Mathematik (Crelles Journal), 2022 (2022), pp.~247--279.

\bibitem{andrews2022extrinsic}
{\sc B.~Andrews, B.~Chow, C.~Guenther, and M.~Langford}, {\em Extrinsic geometric flows}, vol.~206, American mathematical society, 2022.

\bibitem{brendle2016embedded}
{\sc S.~Brendle}, {\em Embedded self-similar shrinkers of genus 0}, Annals of Mathematics,  (2016), pp.~715--728.

\bibitem{cao2013gap}
{\sc H.~Cao and H.~Li}, {\em A gap theorem for self-shrinkers of the mean curvature flow in arbitrary codimension}, Calculus of Variations and Partial Differential Equations, 46 (2013), pp.~879--889.

\bibitem{cao2014pinching}
{\sc S.~Cao, H.~Xu, and E.~Zhao}, {\em Pinching theorems for self-shrinkers of higher codimension}, Results in Mathematics, 79 (2024), pp.~1--26.

\bibitem{cheng2015gap}
{\sc Q.-M. Cheng and G.~Wei}, {\em A gap theorem of self-shrinkers}, Transactions of the American Mathematical Society, 367 (2015), pp.~4895--4915.

\bibitem{colding2015rigidity}
{\sc T.~H. Colding, T.~Ilmanen, and W.~P. Minicozzi}, {\em Rigidity of generic singularities of mean curvature flow}, Publications math{\'e}matiques de l'IH{\'E}S, 121 (2015), pp.~363--382.

\bibitem{colding2012generic}
{\sc T.~H. Colding and W.~P. Minicozzi}, {\em Generic mean curvature flow i; generic singularities}, Annals of mathematics,  (2012), pp.~755--833.

\bibitem{ding2011chern's}
{\sc Q.~Ding and Y.~Xin}, {\em {On Chern's problem for rigidity of minimal hypersurfaces in the spheres}}, Advances in Mathematics, 227 (2011), pp.~131--145.

\bibitem{ding2013volume}
\leavevmode\vrule height 2pt depth -1.6pt width 23pt, {\em Volume growth eigenvalue and compactness for self-shrinkers}, Asian Math,  (2013), pp.~443--456.

\bibitem{ding2014rigidity}
\leavevmode\vrule height 2pt depth -1.6pt width 23pt, {\em The rigidity theorems of self-shrinkers}, Transactions of the American Mathematical Society, 366 (2014), pp.~5067--5085.

\bibitem{drugan2017immersed}
{\sc G.~Drugan and S.~J. Kleene}, {\em Immersed self-shrinkers}, Transactions of the American Mathematical Society, 369 (2017), pp.~7213--7250.

\bibitem{guang2016self}
{\sc Q.~Guang}, {\em Self-shrinkers and translating solitons of mean curvature flow}, PhD thesis, Massachusetts Institute of Technology, 2016.

\bibitem{huisken1990asymptotic}
{\sc G.~Huisken}, {\em Asymptotic behavior for singularities of the mean curvature flow}, Journal of Differential Geometry, 31 (1990), pp.~285--299.

\bibitem{huisken1993hypersurfaces}
\leavevmode\vrule height 2pt depth -1.6pt width 23pt, {\em Local and global behaviour of hypersurfaces moving by mean curvature}, Proceedings of the Summer Research Institute,  (1993), pp.~175--191.

\bibitem{le2011blow}
{\sc N.~Q. Le and N.~Sesum}, {\em Blow-up rate of the mean curvature during the mean curvature flow and a gap theorem for self-shrinkers}, Communications in Analysis and Geometry, 19 (2011), pp.~633--659.

\bibitem{lee2021compactness}
{\sc T.-K. Lee}, {\em Compactness and rigidity of self-shrinking surfaces}, arXiv preprint arXiv:2108.03919,  (2021).

\bibitem{an1992intrinsic}
{\sc A.~Li and J.~Li}, {\em An intrinsic rigidity theorem for minimal submanifolds in a sphere}, Archiv der Mathematik, 58 (1992), pp.~582--594.

\bibitem{Hongwei18shrinker}
{\sc L.~Li, H.~Xu, and Z.~Xu}, {\em A new pinching theorem for complete self-shrinkers and its generalization}, Science China Mathematics, 63 (2020), pp.~1139--1152.

\bibitem{li2012geometric}
{\sc P.~Li}, {\em Geometric analysis}, Cambridge Studies in Advanced Mathematics/Cam-bridge University Press, 134 (2012).

\bibitem{li1982new}
{\sc P.~Li and S.-T. Yau}, {\em A new conformal invariant and its applications to the willmore conjecture and the first eigenvalue of compact surfaces}, Inventiones Mathematicae, 69 (1982), pp.~269--291.

\bibitem{mantegazza2011lecture}
{\sc C.~Mantegazza}, {\em Lecture notes on mean curvature flow}, vol.~290, SpringerScience\& Business Media, 2011.

\bibitem{otsuki1970minimal}
{\sc T.~Otsuki}, {\em Minimal hypersurfaces in a riemannian manifold of constant curvature}, American Journal of Mathematics, 92 (1970), pp.~145--173.

\bibitem{peng1983minimal}
{\sc C.-K. Peng and C.-L. Terng}, {\em Minimal hypersurfaces of spheres with constant scalar curvature}, in Seminar on minimal submanifolds, vol.~103, Princeton Univ. Press Princeton, NJ, 1983, pp.~177--198.

\bibitem{simons1968minimal}
{\sc J.~Simons}, {\em {Minimal varieties in Riemannian manifolds}}, Annals of Mathematics, 88 (1968), pp.~62--105.

\bibitem{lei2017chern}
{\sc H.~Xu and Z.~Xu}, {\em {On Chern's conjecture for minimal hypersurfaces and rigidity of self-shrinkers}}, ournal of Functional Analysis, 273 (2017), pp.~3406--3425.

\bibitem{yang1998chern}
{\sc H.~Yang and Q.-M. Cheng}, {\em {Chern's conjecture on minimal hypersurfaces}}, Mathematische Zeitschrift, 227 (1998), pp.~377--390.

\end{thebibliography}
 \end{document}